\documentclass[12pt,twoside,a4paper]{amsart}
\usepackage[english]{babel}
\usepackage[T1]{fontenc}
\usepackage{amsmath,amssymb,amsthm,verbatim}
\usepackage[left=3cm,top=3.5cm,right=3cm,bottom=3.5cm,bindingoffset=1cm]{geometry}

\usepackage{hyperref}
\hypersetup{final,colorlinks=true,
linkcolor=[rgb]{0.1, 0.1, 0.44},
citecolor=[rgb]{0.1, 0.1, 0.44}}

\theoremstyle{plain}
\newtheorem{theorem}{Theorem}
\newtheorem{lemma}[theorem]{Lemma}
\newtheorem{corollary}[theorem]{Corollary}
\newtheorem{proposition}[theorem]{Proposition}

\theoremstyle{definition}
\newtheorem{definition}{Definition}

\newtheorem{claim}{Claim}[theorem]

\newtheorem{question}{Question}

\newcommand{\N}{\mathbb N}

\newcommand{\Can}{2^{\N}}
\newcommand{\card}{{\rm{card}}}
\newcommand{\fhi}{\varphi}
\newcommand{\quot}[2]{{\raisebox{.2em}{$#1\!$}\left/\raisebox{-.2em}{$#2$}\right.}}

\usepackage{tikz}
\newcommand{\harpoone}[1]{%
\begin{tikzpicture}[#1]%
\draw (0,-0.5pt) -- (0,5pt);%
\draw (0,5pt) -- (1.5pt,3.5pt);%
\end{tikzpicture}%
}
\newcommand{\restr}[1]{\, _{\harpoone{}\, #1}}

\begin{document}

\title[Quotients of projective Fra\"iss\'e limits]{Arcs, hypercubes, and graphs as quotients of projective Fra\"iss\'e limits}
\author[Gianluca Basso \and Riccardo Camerlo]{Gianluca Basso* \and Riccardo Camerlo** }

\newcommand{\acr}{\newline\indent}

\address{\llap{*\,}D\'epartement des syst\`emes d'information\acr
                   Universit\'e de Lausanne\acr
                   Quartier UNIL-Dorigny B\^atiment Internef\acr
                   1015 Lausanne\acr
                   SWITZERLAND}
\email{gianluca.basso@unil.ch}

\address{\llap{**\,}Dipartimento di scienze matematiche \guillemotleft{Joseph-Louis Lagrange}\guillemotright\acr
                    Politecnico di Torino\acr
                    Corso Duca degli Abruzzi, 24\acr
                    10129 Torino\acr
                    ITALY}
\email{riccardo.camerlo@polito.it}


\thanks{The research presented in this paper has been done while the second author was visiting the Department of information systems of the University of Lausanne. He wishes to thank the \'Equipe de logique, and in particular its director prof. Jacques Duparc, for providing such a friendly environment for research.}

\begin{abstract}
We establish some basic properties of quotients of projective Fra\"iss\'e limits and exhibit some classes of compact metric spaces that are the quotient of a projective Fra\"iss\'e limit of a projective Fra\"iss\'e family in a finite language.
We prove the result for the arcs directly, and by applying some closure properties we obtain all hypercubes and graphs as well.
\end{abstract}

\maketitle

\section{Introduction}
Projective Fra\"iss\'e families of topological structures and there limits --- called projective Fra\"iss\'e limits -- for a given language $ \mathcal L $ have been introduced by T. Irwin and S. Solecki in \cite{Irwin2006}.
In that paper, the authors focused on a particular example, where $ \mathcal L =\{ R\}$ contained a unique binary relation symbol and the limit $ \mathbb P =( \mathbb P ,R^{ \mathbb P })$ turned out to be endowed with an equivalence relation $R^{ \mathbb P }$, with the quotient $ \quot{ \mathbb P }{R^{ \mathbb P }} $ being a pseudo-arc.
The characterisation of all spaces that can be obtained, up to homeomorphism, as quotients $ \quot{ \mathbb P }{R^{ \mathbb P }} $, where $( \mathbb P ,R^{ \mathbb P})$ is the projective Fra\"iss\'e limit of a projective Fra\"iss\'e family of finite topological $ \mathcal L $-structures, for $ \mathcal L $ as above and $R^{ \mathbb P }$ an equivalence relation, has been settled in \cite{Camerl2010} and consists of the following list:
\begin{itemize}
\item Cantor space;
\item disjoint sums of $m$ singletons and $n$ pseudo-arcs, with $m+n>0$;
\item disjoint sums of $n$ spaces each of the form $X=P\cup\bigcup_{j\in \N }Q_j$, where $P$ is a pseudo-arc, each $Q_j$ is a Cantor space which is clopen in $X$ and $\bigcup_{j\in \N }Q_j$ is dense in $X$.
\end{itemize}
The purpose of this note is to begin a systematic study of which spaces can be obtained as quotients of projective Fra\"iss\'e limits of projective Fra\"ss\'e families for more general languages.
In section \ref{defandnot} we recall some basic definitions and fix some terminology and notations to be used in the rest of the paper.
In section \ref{facts} we show some simple facts to be used later and note that, if we admit infinite languages, then every compact metric space can be obtained as a quotient of a projective Fra\"iss\'e limit.
We therefore restrict our attention to finite languages and prove, in section \ref{sum-prod}, that the class of spaces that can be obtained as quotients of projective Fra\"iss\'e limits is closed under finite disjoint unions, finite products, and particular quotients satisfying some extra technical conditions.
Section \ref{graphs} presents some examples: after showing that arcs can be obtained as quotients of projective Fra\"iss\'e limits, the results of section \ref{sum-prod} allow to extend this property to hypercubes and graphs.
Finally, in section \ref{qs} we discuss some questions that appear naturally and that could lead to further research in the subject.

We note here that the construction of the projective Fra\"iss\'e limit has already been successfully employed in the literature, leading to interesting results on compact metric spaces and groups of homeomorphisms (see for example \cite{Kwiatk2012}, \cite{Kwiatk2014}, \cite{Bartos2015}).

\section{Basic terminology and definitions} \label{defandnot}
We recall here some basic definitions, mainly from \cite{Irwin2006}, \cite{Camerl2010}.

Let a first order language $ \mathcal L $ be given.
A {\it topological} $ \mathcal L $-{\it structure} is a zero-dimensional, (Hausdorff) compact, second countable space that is also an $ \mathcal L $-structure such that:
\begin{itemize}
\item the interpretations of the relation symbols are closed sets;
\item the interpretations of the function symbols are continuous functions.
\end{itemize}
An {\it epimorphism} between topological $ \mathcal L $-structures $A,B$ is a continuous surjection $ \fhi :A\to B$ such that:
\begin{itemize}
\item $r^B=\underset{n \text{ times}}{\underbrace{ \fhi \times \ldots \times \fhi}}\,(r^A)$ for every $n$-ary relation symbol $r$;
\item $f^B( \fhi (a_1),\ldots , \fhi (a_n))= \fhi f^A(a_1,\ldots ,a_n)$ for every $n$-ary function symbol $f$ and $a_1,\ldots ,a_n\in A$;
\item $ \fhi (c^A)=c^B$ for every constant symbol $c$.
\end{itemize}
An {\it isomorphism} is a bijective epimorphism, so in particular it is an homeomorphism between the supports.
An epimorphism $ \fhi :A\to B$ {\it refines} a covering $ \mathcal U $ of $A$ if the preimage of any element of $B$ is included in some element of $ \mathcal U $.

A family $ \mathcal F $ of topological $ \mathcal L $-structures is a {\it projective Fra\"iss\'e family} if the following properties hold:
\begin{itemize}
\item[(JPP):] (joint projection property) for every $D,E\in \mathcal F $ there are $F\in \mathcal F $ and epimorphisms $F\to D$, $F\to E$;
\item[(AP):] (amalgamation property) for every $C,D,E\in \mathcal F $ and epimorphisms $ \fhi_1:D\to C$, $ \fhi_2:E\to C$ there are $F\in \mathcal F $ and epimorphisms $\psi_1:F\to D$, $\psi_2:F\to E$ such that $ \fhi_1\psi_1= \fhi_2\psi_2$.
\end{itemize}
Given a family $ \mathcal F $ of topological $ \mathcal L $-structures, a topological $ \mathcal L $-structure $ \mathbb F $ is a {\it projective Fra\"iss\'e limit} of $ \mathcal F $ if the following hold:
\begin{itemize}
\item[(L1):] (projective universality) for every $D\in \mathcal F $ there is some epimorphism $ \mathbb F \to D$;
\item[(L2):] for every finite discrete topological space $A$ and continuous function $f: \mathbb F \to A$, there are $D\in \mathcal F $, an epimorphism $ \fhi : \mathbb F \to D$ and a function $f':D\to A$ such that $f=f' \fhi $;
\item[(L3):] (projective ultrahomogeneity) for every $D\in \mathcal F $ and epimorphisms $ \fhi_1, \fhi_2: \mathbb F \to D$ there exists an isomorphism $\psi : \mathbb F \to \mathbb F $ such that $ \fhi_2= \fhi_1\psi $.
\end{itemize}
Property (L2) is equivalent to
\begin{itemize}
\item[(L2$'$):] for any clopen covering $ \mathcal U $ of $ \mathbb F $ there are $D\in \mathcal F $ and an epimorphism $ \mathbb F \to D$ refining $ \mathcal U $.
\end{itemize}
In \cite{Irwin2006} it is proved that every non-empty, at most countable, projective Fra\"iss\'e family of finite topological $ \mathcal L $-structures has a projective Fra\"iss\'e limit, which is unique up to isomorphism.

If $ \mathcal F $ is a class of topological $ \mathcal L $-structures, a {\it fundamental sequence} $(D_n,\pi_n)$ is a sequence of elements of $ \mathcal F $ together with epimorphisms $\pi_n:D_{n+1}\to D_n$ such that, denoting $\pi_n^m=\pi_n\cdots\pi_{m-1}:D_m\to D_n$ for $n<m$ and letting $\pi_n^n:D_n\to D_n$ be the identity, the following properties hold:
\begin{itemize}
\item for every $D\in \mathcal F $ there are $n$ and an epimorphism $D_n\to D$;
\item for any $n$, any $E,F\in \mathcal F $ and any epimorphisms $ \fhi_1 :F\to E$, $ \fhi_2:D_n\to E$, there exist $m\geq n$ and an epimorphism $\psi :D_m\to F$ such that $ \fhi_1\psi= \fhi_2\pi_n^m$.
\end{itemize}

To study projective Fra\"iss\'e limits it is enough to consider fundamental sequences, due to the following fact whose details can be found in \cite{Camerl2010}.

\begin{proposition}
Let $ \mathcal F $ be a non-empty, at most countable projective Fra\"ss\'e family of finite topological $ \mathcal L $-structures.
Then the following are equivalent.
\begin{enumerate}
\item $ \mathcal F $ is a projective Fra\"iss\'e family;
\item $ \mathcal F $ has a projective Fra\"iss\'e limit;
\item $ \mathcal F $ has a fundamental sequence.
\end{enumerate}
Moreover, in this case the projective Fra\"iss\'e limits of $ \mathcal F $ and of its fundamental sequence coincide.
A limit for both is the inverse limit of the fundamental sequence.
\end{proposition}

In the sequel, whenever we denote a language with a subscript, like in $ \mathcal L_R$, we mean that the language contains a distinguished binary relation symbol represented in the subscript (in this case the symbol $R$).
The following definition is central.

\begin{definition}
A compact metric space $X$ is $ \mathcal L_R$-{\it representable} if there exists an at most countable projective Fra\"iss\'e family of finite topological $ \mathcal L_R $-structures such that, denoting $ \mathbb F $ its projective Fra\"iss\'e limit, $R^{ \mathbb F }$ is an equivalence relation and $ \quot{ \mathbb F }{R^{ \mathbb F }} $ is homeomorphic to $X$.

Space $X$ is {\it finitely representable} if it is $ \mathcal L_R$-representable for some finite $ \mathcal L_R$.
\end{definition}

In this terminology, when $ \mathcal L_R=\{ R\} $, the $ \mathcal L_R$-representable spaces have been characterised in \cite{Camerl2010}.

\section{Some preliminary facts} \label{facts}
In this section we collect some basic properties of projective Fra\"iss\'e families and their limits.

\begin{proposition} \label{arflessivsimetric}
Let $ \mathcal L_R=\{ R,\ldots\} $ and suppose $ \mathcal F $ is a projective Fra\"iss\'e family in the language $ \mathcal L_R$.
Let $ \mathbb F =( \mathbb F ,R^{ \mathbb F },\ldots )$ be the projective Fra\"iss\'e limit of $ \mathcal F $.
\begin{enumerate}
\item If $R$ is interpreted by all structures in $ \mathcal F $ as a reflexive relation, then $R^{ \mathbb F }$ is reflexive as well.
\item If $R$ is interpreted by all structures in $ \mathcal F $ as a symmetric relation, then $R^{ \mathbb F }$ is symmetric as well.
\item If $R$ is interpreted by all structures in $ \mathcal F $ as an anti-symmetric relation, then $R^{ \mathbb F }$ is anti-symmetric as well.
\item If $R$ is interpreted by all structures in $ \mathcal F $ as a transitive relation, then $R^{ \mathbb F }$ is transitive as well.
\item If $R$ is interpreted by all structures in $ \mathcal F $ as a total relation, then $R^{ \mathbb F } $ is total as well.
\item If $R$ is interpreted by all structures in $ \mathcal F $ as having a first (respectively, last) element, then $R^{ \mathbb F }$ has a first (respectively, last) element as well.
\item If $R$ is interpreted by all strucutres in $ \mathcal F $ as a connected relation, then for any partition $\{ U,V\} $ of $ \mathbb F $ into clopen sets there are $x\in U,y\in V$ with $xR^{ \mathbb F }y$.
\end{enumerate}
\end{proposition}

\begin{proof}
In this proof we will use property (L2$'$) extensively.

(1) and (2)
The proof is similar to the argument carried out in the proof of \cite[lemma 4.1]{Irwin2006}.

(3) Let $x,y\in \mathbb F $ be distinct elements such that $xR^{ \mathbb F }yR^{ \mathbb F }x$.
Pick a clopen subset $U$ of $ \mathbb F $ such that $x\in U,y\notin U$ and find $A\in \mathcal F $ with an epimorphism $\varphi : \mathbb F \to A$ refining $\{ U, \mathbb F \setminus U\} $.
Then $\varphi (x),\varphi (y)$ are distinct and $\varphi (x)R^A\varphi (y)R^A\varphi (x)$.

(4)
Let $x,y,z\in \mathbb F $, with $xR^{ \mathbb F }yR^{ \mathbb F }z$.
Since $R^{ \mathbb F } $ is closed, it is enough to show that for any neighbourhoods $U$ of $x$ and $V$ of $z$ there are $x'\in U$, $z'\in V$ with $x'R^{ \mathbb F }z'$.
Let $U'\subseteq U,V'\subseteq V$ be clopen neighbourhoods of $x,z$, respectively, with $U'=V'$ if $x=z$ and $U'\cap V'=\emptyset $ otherwise.
Let $A\in \mathcal F$ and $\varphi : \mathbb F \to A$ be an epimorphism refining the clopen covering $\{ U',V', \mathbb F \setminus (U'\cup V')\} $.
Since $\varphi (x)R^A\varphi (z)$, there are $x'\in U'$, $z'\in V'$ with $\varphi (x)=\varphi (x')$, $\varphi (z)=\varphi (z')$, $x'R^{ \mathbb F }z'$.

(5)
It is enough to show that, given $x,y\in \mathbb F $, whenever $U,V$ are clopen neighbourhoods of $x,y$ respectively, there are $x'\in U$, $y'\in V$ such that either $x'R^{ \mathbb F }y'$ or $y'R^{ \mathbb F }x'$.
Moreover, if $x=y$ it can be assumed that $U=V$, while for $x\neq y$ one can take $U\cap V=\emptyset $.
Let $A\in \mathcal F $ with an epimorphism $\varphi : \mathbb F \to A$ refining the clopen covering $\{ U,V, \mathbb F \setminus (U\cup V)\} $.
Since $\varphi (x)R^A\varphi (y)$ or $\varphi (y)R^A\varphi (x)$, there are $x'\in U$, $y'\in V$ such that $\varphi (x')=\varphi (x)$, $\varphi (y')=\varphi (y)$ and either $x'R^{ \mathbb F }y'$ or $y'R^{ \mathbb F }x'$.

(6)
Argue for the first element, the situation for the last being similar.
Fix a compatible complete metric on $ \mathbb F $ and, for each positive integer $n$, let $ \mathcal U_n$ be a partition of $ \mathbb F $ with clopen sets of diameter less than $ \frac 1n $ such that $ \mathcal U_{n+1}$ refines $ \mathcal U_n$.
Let $\varphi_n: \mathbb F \to A_n$ be an epimorphism refining $ \mathcal U_n$ onto some $A_n\in \mathcal F $.
Let $x_n\in \mathbb F $ be such that $\varphi_n(x_n)$ is the first element of $R^{A_n}$ and fix a limit point $x$ of the sequence $x_n$, in order to show that $\forall y\in \mathbb F \ xR^{ \mathbb F }y$.
For this it is enough to prove that given clopen neighbourhoods $U,V$ of $x,y$, respectively, where it can be assumed that $U=V$ if $x=y$ and that $U\cap V=\emptyset $ if $x\neq y$, there are $x'\in U$, $y'\in V$ with $x'R^{ \mathbb F }y'$.
Take $n$ such that if $x\in W\in \mathcal U_n$ and $y\in W'\in \mathcal U_n$, then $W\subseteq U,W'\subseteq V$.
Let $n'\geq n$ be such that $x_{n'}\in W$.
Notice that $\varphi_{n'}$ refines $\{ W,W', \mathbb F \setminus (W\cup W')\} $.
Since $\varphi_{n'}(x_{n'})R^{A_{n'}}\varphi (y)$, there are $x'\in W$, $y'\in W'$ such that $\varphi_{n'}(x')=\varphi_{n'}(x_{n'})$, $\varphi_{n'}(y')=\varphi_{n'}(y)$, $x'R^{ \mathbb F }y'$.

(7)
As for the argument in the proof of \cite[lemma 4.3]{Irwin2006}.
\end{proof}

Notice that for (1),(2),(5),(6),(7) the converse holds as well.

\begin{proposition} \label{machrelassional}
Let $ \mathcal L $ be a language and let $ \mathcal L'$ be obtained from $ \mathcal L $ by replacing each constant symbol $c$ with a unary relation symbol $R_c$ and each function symbol $f$, say with $m$ arguments, with a new relation symbol $R_f$, with $m+1$ arguments.
For every $ \mathcal L $ structure $A$, let $A'$ be the $ \mathcal L'$ structure defined as follows.
\begin{itemize}
\item $A'$ has the same universe as $A$;
\item the interpretations in $A'$ of the relation symbols of $ \mathcal L $ are the same as in $A$;
\item if $c$ is a constant symbol of $ \mathcal L $, then $R_c^{A'}(a)\Leftrightarrow c^A=a$;
\item if $f$ is an $m$-ary function symbol of $ \mathcal L $, then $R_f^{A'}$ is the graph of $f^A$.
\end{itemize}
Then
\begin{enumerate}
\item $ \fhi :A\to B$ is an $ \mathcal L$-epimorphism if and only if $ \fhi :A'\to B'$ is an $ \mathcal L'$-epimorphism.
\item If $ \mathcal F $ is a non-empty, at most countable, projective Fra\"iss\'e family of finite topological $ \mathcal L $-structures and $ \mathcal F'$ is the class obtained from $ \mathcal F $ by replacing each $A\in \mathcal F $ with the $ \mathcal L'$-structure $A'$, then $ \mathcal F'$ is a projective Fra\"iss\'e family.

Moreover, if $ \mathbb F $ is a projective Fra\"iss\'e limit of $ \mathcal F $, then a projective Fra\"iss\'e limit $ \mathbb F'$ of $ \mathcal F'$ can be constructed as follows:
\begin{itemize}
\item the universes of $ \mathbb F , \mathbb F'$ are the same;
\item the interpretations of all relation symbols of $ \mathcal L $ are the same;
\item for every constant symbol $c$ of $ \mathcal L $, $R_c^{ \mathbb F'}(x)\Leftrightarrow c^{ \mathbb F }=x$;
\item for every function symbol $f$ of $ \mathcal L $, $R_f^{ \mathbb F'}$ is the graph of $f^{ \mathbb F}$.
\end{itemize}
\end{enumerate}
\end{proposition}

\begin{proof}
\sloppypar
(1) Suppose $ \fhi :A\to B$ is an epimorphism.
Take an $m$-ary function symbol $f$ in $ \mathcal L $.
Then $R_f^{B'}(b_1,\ldots ,b_m,b_{m+1})\Leftrightarrow f^B(b_1,\ldots ,b_m)=b_{m+1}\Leftrightarrow\exists a_1,\ldots ,a_m,a_{m+1}( \fhi (a_1)=b_1\wedge\ldots\wedge \fhi (a_m)=b_m\wedge \fhi (a_{m+1})=b_{m+1}\wedge f^A(a_1,\ldots ,a_m)=a_{m+1})\Leftrightarrow\exists a_1,\ldots ,a_{m+1}( \fhi (a_1)=b_1\wedge\ldots\wedge \fhi (a_{m+1})=b_{m+1}\wedge R_f^{A'}(a_1,\ldots ,a_{m+1}))$.
Similarly, if $c\in \mathcal L $ is a constant symbol, $R_c^{B'}(b)\Leftrightarrow c^B=b\Leftrightarrow b= \fhi (c^A)\Leftrightarrow\exists a(b= \fhi (a)\wedge R_c^{A'}(a))$.
So $ \fhi :A'\to B'$ is an epimorphism.

Conversely, assume $ \fhi :A'\to B'$ is an epimorphism and consider again an $m$-ary function symbol $f\in \mathcal L $.
Then $f^A(a_1,\ldots ,a_m)=a_{m+1}\Leftrightarrow R_f^{A'}(a_1,\ldots ,a_m,a_{m+1})\Rightarrow R_f^{B'}( \fhi (a_1),\ldots , \fhi (a_m), \fhi (a_{m+1}))\Leftrightarrow f^B( \fhi (a_1),\ldots , \fhi (a_m))= \fhi (a_{m+1})$.
Similarly, for a constant symbol $c\in \mathcal L $, one has $a=c^A\Leftrightarrow R_c^{A'}(a)\Rightarrow R_c^{B'}( \fhi (a))\Leftrightarrow \fhi (a)=c^B$.
Thus $ \fhi :A\to B$ is an epimorphism too.

(2) From (1), if $(D_n)$ is a fundamental sequence for $ \mathcal F $, then $(D'_n)$ is a fundamental sequence for $ \mathcal F'$ endowed with the same family of epimorphisms.
Since the projective Fra\"iss\'e limits can be computed as inverse limits, the result about the universe of the limits and the interpretation of the relation symbols of $ \mathcal L $ follows.
Finally, denote by $\pi_n^{\infty }$ the projections of the limits onto the members of the fundamental sequence.
Let $c$ be a constant symbol of $ \mathcal L $; then $R_c^{ \mathbb F'}(x)\Leftrightarrow\forall n\in \N \ R^{D'_n}(\pi_n^{\infty }(x))\Leftrightarrow\forall n\in\N \ c^{D_n}=\pi_n^{\infty }(x)\Leftrightarrow c^{ \mathbb F }=x$.
Similarly, for an $m$-ary function symbol $f$ of $ \mathcal L $, one has $R_f^{ \mathbb F'}(x_1,\ldots ,x_m,x_{m+1})\Leftrightarrow\forall n\in \N \ R_f^{D'_n}(\pi_n^{\infty }(x_1),\ldots ,\pi_n^{\infty }(x_m),\pi_n^{\infty }(x_{m+1}))\Leftrightarrow\forall n\in \N \ f^{D_n}(\pi_n^{\infty }(x_1),\ldots ,\pi_n^{\infty }(x_m))=\pi_n^{\infty }(x_{m+1})\Leftrightarrow f^{ \mathbb F }(x_1,\ldots ,x_m)=x_{m+1}$.
\end{proof}

The import of proposition \ref{machrelassional} is that for our purposes we are always allowed to consider only relational languages.
For later use, note that it follows that the groups of isomorphisms of $ \mathbb F $ and of $ \mathbb F'$ coincide.

Now we show that if one admits infinite languages, then every compact metric space is homeomorphic to the quotient of a projective Fra\"iss\'e limit.
Consequently, in the sequel we will be interested in studying what kind of spaces can be obtained with finite languages.

\begin{lemma} \label{rigidity}
Let $ \mathcal L $ be any language and let $D_0,D_1,\ldots $ be a family of finite topological $ \mathcal L $-structures.
If for every $n\leq m$ there is exactly one epimorphism $\pi_n^m:D_m\to D_n$, then $(D_n,\pi_n^{n+1})$ is a fundamental sequence.
\end{lemma}

\begin{proof}
First notice that from the hypothesis it follows that for any $n,m$ there is at most one epimorphism $D_m\to D_n$.
If $n\leq m$ this is in the hypothesis; on the other hand, if $n>m$, the existence of an epimorphism $ \fhi :D_m\to D_n$ implies that $ \fhi $ and $\pi_m^n$ are actually isomorphisms; if there were two different isomorphisms $D_m\to D_n$, their compositions with $\pi_m^n$ would yield two different isomorphisms $D_n\to D_n$.

Consequently, given any two epimorphisms $ \fhi_1:D_h\to D_k$, $ \fhi_2:D_p\to D_k$ and letting $m=\max (h,p)$, one has $ \fhi_1\pi_h^m= \fhi_2\pi_p^m$.
\end{proof}

\begin{proposition}
Let $ \mathcal L_R=\{ R,\rho_s\}_{s\in 2^{<\omega }}$, where the $\rho_s$ are unary relation symbols for all $s\in 2^{<\omega }$.
Then every compact metric space is $ \mathcal L_R$-representable.
\end{proposition}

\begin{proof}
Let $X$ be a compact metric space and let $\equiv $ be a closed equivalence relation on $ \Can $ such that $X\backsimeq \quot{ \Can }{\equiv }$.
Define $ \mathcal L_R$-structures $D_n=(2^n,R^{D_n},\rho_s^{D_n})_{s\in 2^{<\omega }}$ by letting
\[
\begin{array}{lcl}
uR^{D_n}u' & \Leftrightarrow & \exists x,x'\in \Can\ (u\subseteq x\wedge u'\subseteq x'\wedge x\equiv x') \\
\rho_s^{D_n}(u) & \Leftrightarrow & s\subseteq u\vee u\subseteq s
\end{array}
\]
Now notice that, given $n\leq m$, the only epimorphism $D_m\to D_n$ is the restriction map $\pi_n^m$ defined by $\pi_n^m(w)=w \restr n$.
Indeed, if $w,w'\in 2^m$ are such that $wR^{D_m}w'$, let $x,x'\in \Can $ with $w\subseteq x$, $w'\subseteq x'$, $x\equiv x'$; since $x \restr n=\pi_n^m(w)$, $x' \restr n=\pi_n^m(w')$, it follows that $\pi_n^m(w)R^{D_n}\pi_n^m(w')$.
Moreover, if $w\in 2^m$ satisfies $\rho_s^{D_m}(w)$ for some $s\in 2^{<\omega }$, so that $w$ is compatible with $s$, its restriction $w \restr n$ is compatible with $s$ as well, so $\rho_s^{D_n}(\pi_n^m(w))$ holds.
Conversely, assume first that $u,u'\in 2^n$ fulfill $uR^{D_n}u'$ and let $x,x'\in \Can $ such that $u\subseteq x$, $u'\subseteq x'$, $x\equiv x'$; then $x \restr mR^{D_m}x' \restr m$, $\pi_n^m(x \restr m)=u$, $\pi_n^m(x' \restr m)=u'$.
Finally, suppose that $s\in 2^{<\omega },u\in 2^n$ are such that $\rho_s^{D_n}(u)$; then there is at least an element $w\in 2^m$ such that $\pi_n^m(w)=w \restr n=u$ and $w$ is compatible with $s$, so that $\rho_s^{D_m}(w)$.
To see that $\pi_n^m$ is the unique epimorphism $D_m\to D_n$, notice that for any $w\in 2^m$, the unique element $u\in 2^n$ such that $\rho_w^{D_n}(u)$ is $w \restr n$.

Consequently, by lemma \ref{rigidity}, $(D_n,\pi_n^{n+1})$ is a fundamental sequence.
Let $ \mathbb F =( \Can ,R^{ \mathbb F },\rho_s^{ \mathbb F })_{s\in 2^{<\omega }}$ be its inverse limit.
It is now enough to prove $R^{ \mathbb F }= {\equiv } $, so let $x,x'\in \Can $.
If $x\equiv x'$, then $\forall  n\in \N \ x \restr n R^{D_n}x' \restr n $, so that $xR^{ \mathbb F }x'$.
Conversely, if $xR^{ \mathbb F }x'$, so that $\forall n\in \N \ x \restr n R^{D_n}x' \restr n $, for every $n\in \N $ there are $x_n,x'_n\in \Can $ such that $x \restr n \subseteq x_n$, $x' \restr n\subseteq x'_n$, $x_n\equiv x'_n$, so that $\lim_{n\rightarrow\infty }x_n=x$, $\lim_{n\rightarrow\infty }x'_n=x'$, $x\equiv x'$, since $\equiv $ is closed.
\end{proof}

\section{Closure under topological operations} \label{sum-prod}
This section collects some closure properties of finitely representable spaces.
We will need the following notion.

\begin{definition}
Let $ \mathbb F $ be a projective Fra\"iss\'e limit of a projective Fra\"iss\'e family of topological structures in some language $ \mathcal L_R$ and suppose that $R^{ \mathbb F }$ is an equivalence relation.
A point $x\in \mathbb F $ is {\it almost stable} if, for all isomorphisms $ \fhi $ of $ \mathbb F $, one has that $ \fhi (x)R^{ \mathbb F }x$.
\end{definition}

Notice that the set of almost stable points is invariant under the equivalence relation $R^{ \mathbb F }$.

\begin{theorem} \label{closurefinitesums}
The finite disjoint sum of finitely representable spaces is finitely representable.
\end{theorem}

\begin{proof}
It is enough to prove the result for the disjoint sum of two spaces\footnote{Notice that for the sum of $n$ spaces, a direct proof would provide a smaller language than the one resulting by iterating the construction in the proof.}.
So, for $i\in\{ 1,2\} $ let $X_i$ be $ \mathcal L_R^i$-representable for some finite $ \mathcal L_R^i$, as witnessed by a projective Fra\"iss\'e family $ \mathcal F_i$ with limit $ \mathbb F_i$.
By proposition \ref{machrelassional} one can assume that each $ \mathcal L_R^i$ is a relational language, and moreover $ \mathcal L_R^1\cap \mathcal L_R^2=\{ R\} $.
Let $ \mathcal L_R= \mathcal L_R^1\cup \mathcal L_R^2\cup\{ P_1,P_2\} $, where $P_1,P_2$ are new unary relation symbols.

Given topological $ \mathcal L_R^i$-structures $A_i$, for $i\in\{ 1,2\} $, define an $ \mathcal L_R$-structure $A=A_1\oplus A_2$ as follows:

\begin{itemize}
\item $A$ is a disjoint union $A_1\cup A_2$, with each $A_i$ clopen in $A$;
\item $R^A=R^{A_1}\cup R^{A_2}$;
\item $P_i^A=A_i$;
\item if $S\in \mathcal L_R^i$ is a relation symbol different from $R$, then $S^A=S^{A_i}$.
\end{itemize}
Notice that if $ \fhi_i:A_i\to B_i$ are $ \mathcal L_R^i$-epimorphisms for $i\in\{ 1,2\} $, then $ \fhi_1\cup \fhi_2:A_1\oplus A_2\to B_1\oplus B_2$ is an $ \mathcal L_R$-epimorphism.
Conversely, if $ \fhi :A_1\oplus A_2\to B_1\oplus B_2$ is an $ \mathcal L_R$-epimorphism, then by the interpretations of symbols $P_1,P_2$, the restriction $ \fhi_i$ of $ \fhi $ to $A_i$ has range included in --- in fact, equal to --- $B_i$; moreover $ \fhi_i:A_i\to B_i$ is an $ \mathcal L_R^i$-epimorphism.

Define $ \mathcal F $ as the class of $ \mathcal L_R$-structures $A=(A,R^A,\ldots ,P_1^A,P_2^A)$ of the form $A=A_1\oplus A_2$, where $A_i\in \mathcal F_i$.

\begin{claim}
$ \mathcal F $ is a projective Fra\"iss\'e family.
\end{claim}

{\it Proof of claim.}
JPP:
Let $A=A_1\oplus A_2$, $B=B_1\oplus B_2\in \mathcal F $.
By (JPP) of $ \mathcal F_i$, let $C_i\in \mathcal F_i$, with epimorphisms $ \fhi_i:C_i\to A_i$, $\psi_i:C_i\to B_i$.
Set $C=C_1\oplus C_2\in \mathcal F $, $ \fhi = \fhi_1\cup \fhi_2:C\to A$, $\psi =\psi_1\cup\psi_2:C\to B$.
Then $ \fhi ,\psi $ are epimorphisms.

AP:
Let $A=A_1\oplus A_2$, $B=B_1\oplus B_2$, $C=C_1\oplus C_2\in \mathcal F $, with epimorphisms $ \fhi :B\to A$, $\psi :C\to A$.
So let $ \fhi_i= \fhi \restr{B_i} $, $\psi_i=\psi \restr{C_i} $, then $ \fhi_i:B_i\to A_i$, $\psi_i:C_i\to A_i$ are epimorphisms.
By (AP) for $ \mathcal F_i$, let $D_i \in \mathcal F_i$, $ \fhi'_i:D_i\to B_i$, $\psi'_i:D_i\to C_i$ be epimorphisms such that $ \fhi_i\fhi'_i=\psi_i\psi'_i$.
Let $D=D_1\oplus D_2\in \mathcal F $.
So $ \fhi'= \fhi'_1\cup \fhi'_2:D\to B$, $\psi'=\psi'_1\cup\psi'_2:D\to C$ are epimorphisms such that $ \fhi \fhi'=\psi\psi'$.
\qed

\medskip
Let $ \mathbb F = \mathbb F_1\oplus \mathbb F_2$.

\begin{claim}
$ \mathbb F $ is the projective Fra\"iss\'e limit of $ \mathcal F $.
\end{claim}

{\it Proof of claim.}
It is enough to carry out the following three verifications:
\begin{itemize}
\item (L1) Let $A=A_1\oplus A_2\in \mathcal F $.
By projective universality of $ \mathbb F_i$, let $ \fhi_i: \mathbb F_i\to A_i$ be an epimorphism.
Then $ \fhi = \fhi_1\cup \fhi_2: \mathbb F \to A$ is an epimorphism.
\item (L2$'$) Let $ \mathcal U $ be a partition of $ \mathbb F $ into clopen sets, which can be assumed to refine $\{ \mathbb F_1, \mathbb F_2\} $.
So $ \mathcal U \cap \mathcal P ( \mathbb F_i)$ is a partition of $ \mathbb F_i$ into clopen sets.
Let $A_i\in \mathcal F_i$ with an epimorphism $ \fhi_i: \mathbb F_i\to D_i$ refining $ \mathcal U \cap \mathcal P ( \mathbb F_i)$.
So $ \fhi = \fhi_1\cup \fhi_2: \mathbb F \to A_1\oplus A_2$ is an epimorphism refining $ \mathcal U $.
\item (L3) Let $A=A_1\oplus A_2\in \mathcal F $, with epimorphisms $ \fhi_1,\fhi_2: \mathbb F \to A$.
So $ \fhi_j \restr{ \mathbb F_i} $ are epimorphisms $ \mathbb F_i\to A_i$.
By projective ultrahomogeneity of $ \mathbb F_i$, let $\psi_i: \mathbb F_i\to \mathbb F_i$ be an isomorphism such that $ \fhi_1 \restr{ \mathbb F_i} \psi_i= \fhi_2 \restr{ \mathbb F_i} $.
So $\psi =\psi_1\cup\psi_2: \mathbb F \to \mathbb F $ is an isomorphism such that $ \fhi_1\psi = \fhi_2$.
\end{itemize}
\qed

Notice that $R^{ \mathbb F }$ is an equivalence relation on $ \mathbb F $ and $ \quot{\mathbb F }{R^{ \mathbb F }} $ is a disjoint sum of $ \quot{\mathbb F_1}{R^{ \mathbb F_1}} , \quot{ \mathbb F_2}{R^{ \mathbb F_2}}$, completing the proof.
\end{proof}

For later use we remark that in the proof of theorem \ref{closurefinitesums}, if $x$ is an almost stable point in one of the $ \mathbb F_i$, then $x$ is almost stable also in the resulting $ \mathbb F $.

\begin{theorem} \label{products}
The finite product of finitely representable spaces is finitely representable.
\end{theorem}

\begin{proof}
It is enough to prove the assertion for products of two factors\footnote{Remarks about the language similar to those in theorem \ref{closurefinitesums} apply here.}.
So, for $i\in\{ 1,2\} $ let $X_i$ be $ \mathcal L_R^i$-representable, for some finite $ \mathcal L_R^i$, as witnessed by a projective Fra\"iss\'e family $ \mathcal F_i$ with limit $ \mathbb F_i$.
By proposition \ref{machrelassional} it can be assumed that $ \mathcal L_R^1, \mathcal L_R^2$ are relational languages, and moreover $ \mathcal L_R^1\cap \mathcal L_R^2=\{ R\} $.
Let $ \mathcal L_R= \mathcal L_R^1\cup \mathcal L_R^2\cup\{ r_1,r_2\} $, where $r_1,r_2$ are two new binary relation symbols.
Let $ \mathcal F =\{ A\times B\mid A\in \mathcal F_1,B\in \mathcal F_2\} $ where:
\begin{itemize}
\item $(a,b)R^{A\times B}(a',b')\Leftrightarrow aR^Aa'\wedge bR^Bb'$;
\item $S^{A\times B}((a_1,b_1),\ldots ,(a_m,b_m))\Leftrightarrow S^A(a_1,\ldots ,a_m)$ for any $m$-ary relation symbol $S\in \mathcal L_R^1\setminus\{ R\} $;
\item $S^{A\times B}((a_1,b_1),\ldots ,(a_m,b_m))\Leftrightarrow S^B(b_1,\ldots ,b_m)$ for any $m$-ary relation symbol $S\in \mathcal L_R^2\setminus\{ R\} $;
\item $r_1^{A\times B}((a_1,b_1),(a_2,b_2))\Leftrightarrow a_1=a_2$;
\item $r_2^{A\times B}((a_1,b_1),(a_2,b_2))\Leftrightarrow b_1=b_2$.
\end{itemize}

\begin{claim}
$ \fhi :A\times B\to C\times D$ is an epimorphism if and only if $ \fhi =\psi\times\theta $ for some epimorphisms $\psi :A\to C,\theta :B\to D$.
\end{claim}

{\it Proof of claim.}
\sloppypar{
Let $ \fhi :A\times B\to C\times D$ be an epimorphism.
Since $r_1^{A\times B}((a,b_1),(a,b_2))$, from $r_1^{C\times D}( \fhi (a,b_1), \fhi (a,b_2))$ it follows that $ \fhi (a,b_1), \fhi (a,b_2)$ have the same first component; similarly for $ \fhi (a_1,b), \fhi (a_2,b)$.
This means that $ \fhi =\psi\times\theta $ for some surjective $\psi :A\to C$, $\theta :B\to D$.
It remains to prove that $\psi $, and similarly $\theta $, are epimorphisms.}

Suppose $cR^Cc'$.
Since $R^D$ is reflexive, by reflexivity of $R^{ \mathbb F}$ and proposition \ref{arflessivsimetric}(1), it follows that for any $d\in D$ one has $(c,d)R^{C\times D}(c',d)$.
So there are $(a,b),(a',b')\in A\times B$ such that $ \fhi (a,b)=(c,d)$, $ \fhi (a',b')=(c',d)$, $(a,b)R^{A\times B}(a',b')$.
Consequently, $\psi (a)=c$, $\psi (a')=c'$, $aR^Aa'$.
Conversely, if $aR^Aa'$, for any $b\in B$ one has $(a,b)R^{A\times B}(a',b)$, whence $(\psi (a),\theta (b))R^{C\times D}(\psi (a'),\theta (b))$, so $\psi(a)R^C\psi (a')$.

Let $S\in \mathcal L_R^1\setminus\{ R\} $ be an $m$-ary relation symbol.
If $S^C(c_1,\ldots ,c_m)$, for any $d\in D$ one has $S^{C\times D}((c_1,d),\ldots ,(c_m,d))$.
Let $a_1,\ldots ,a_m\in A$, $b_1,\ldots ,b_m\in B$ with $ \fhi (a_1,b_1)=(c_1,d),\ldots , \fhi (a_m,b_m)=(c_m,d)$, $S^{A\times B}((a_1,b_1),\ldots ,(a_m,b_m))$.
This implies $\psi (a_1)=c_1$, $\ldots ,\psi (a_m)=c_m$, $S^A(a_1,\ldots ,a_m)$.
Conversely, whenever $S^A(a_1,\ldots ,a_m)$, picking any $b\in B$, one has $S^{A\times B}((a_1,b),\ldots ,(a_m,b))$, whence $S^{C\times D}( \fhi (a_1,b),\ldots , \fhi (a_m,b))$, which allows to conclude that $S^C(\psi (a_1),\ldots ,\psi (a_m))$.

Assume now $\psi :A\to C$, $\theta :B\to D$ are epimorphisms, and set $ \fhi =\psi\times \theta $.
Then, for any $(c,d),(c',d')\in C\times D$,
\[
\begin{split}
(c,d) & R^{C\times D}(c',d')\Leftrightarrow cR^Cc'\wedge dR^Dd'\Leftrightarrow \\
 & \Leftrightarrow\exists a,a'\in A\ \exists b,b'\in B \\
 & \ \ \ (\psi (a)=c\wedge\psi (a')=c'\wedge\theta (b)=d\wedge\theta (b')=d'\wedge aR^Aa'\wedge bR^Bb')\Leftrightarrow \\
 & \Leftrightarrow\exists a,a'\in A\ \exists b,b'\in B \\
 & \ \ \ ( \fhi (a,b)=(c,d)\wedge \fhi (a',b')=(c',d')\wedge (a,b)R^{A\times B}(a',b')).
\end{split}
\]
Moreover, if $S\in \mathcal L_R^1\setminus\{ R\} $ is an $m$-ary relation symbol and $S^{A\times B}((a_1,b_1),\ldots ,(a_m,b_m))$, then $S^A(a_1,\ldots ,a_m)$, whence $S^C(\psi (a_1),\ldots ,\psi (a_m))$ and finally $S^{C\times D}( \fhi (a_1,b_1),\ldots , \fhi (a_m,b_m))$.
Conversely, suppose $S^{C\times D}((c_1,d_1),\ldots ,(c_m,d_m))$, which is equivalent to $S^C(c_1,\ldots ,c_m)$.
So there are $a_1,\ldots ,a_m\in A$ such that $\psi (a_1)=c_1,\ldots ,\psi (a_m)=c_m$, $S^A(a_1,\ldots ,a_m)$.
Taking any $b_1,\ldots ,b_m\in B$ such that $\theta (b_1)=d_1,\ldots ,\theta (b_m)=d_m$, one has $ \fhi (a_1,b_1)=(c_1,d_1),\ldots , \fhi (a_m,b_m)=(c_m,d_m)$, $S^{A\times B}((a_1,b_1),\ldots ,(a_m,b_m))$.
Similarly for symbols in $ \mathcal L_R^2$.
\qed

\begin{claim}
$ \mathcal F $ is a projective Fra\"iss\'e family.
\end{claim}

{\it Proof of claim.}
JPP:
Let $A\times B$, $C\times D\in \mathcal F $.
By (JPP) of $ \mathcal F_1$ and $ \mathcal F_2$, let $E\in \mathcal F_1$, $F\in \mathcal F_2$ with epimorphisms $ \fhi_1:E\to A$, $ \fhi_2:E\to C$, $\psi_1:F\to B$, $\psi_2:F\to D$.
Then $ \fhi_1\times\psi_1:E\times F\to A\times B$, $ \fhi_2\times\psi_2:E\times F\to C\times D$ are epimorphisms.

AP:
Let $A_1\times A_2$, $B_1\times B_2$, $C_1\times C_2\in \mathcal F $ with epimorphisms $ \fhi :B_1\times B_2\to A_1\times A_2$, $\psi :C_1\times C_2\to A_1\times A_2$.
By the preceding claim, there are epimorphims $ \fhi_1, \fhi_2,\psi_1,\psi_2$ such that $ \fhi = \fhi_1\times \fhi_2$, $\psi =\psi_1\times\psi_2$.
Using (AP) of $ \mathcal F_1$, $ \mathcal F_2$, let $D_1\in \mathcal F_1$, $D_2\in \mathcal F_2$ with epimorphisms $\theta_1:D_1\to B_1$, $\rho_1:D_1\to C_1,\theta_2:D_2\to B_2$, $\rho_2:D_2\to C_2$ be such that $ \fhi_1\theta_1=\psi_1\rho_1$, $ \fhi_2\theta_2=\psi_2\rho_2$.
Thus $\theta_1\times\theta_2:D_1\times D_2\to B_1\times B_2$, $\rho_1\times\rho_2:D_1\times D_2\to C_1\times C_2$ are epimorphisms such that $ \fhi (\theta_1\times\theta_2)=\psi (\rho_1\times\rho_2)$.
\qed

\medskip
Let now $(A_n,\pi_n ),(B_n,\rho_n)$ be fundamental sequences for $ \mathcal F_1, \mathcal F_2$, respectively.

\begin{claim}
$(A_n\times B_n,\pi_n\times\rho_n)$ is a fundamental sequence for $ \mathcal F $.
\end{claim}

{\it Proof of claim.}
Let $A\times B\in \mathcal F $.
There are $n,m\in \N $ and epimorphisms $ \fhi:A_n\to A$, $\psi :B_m\to B$.
If $n\leq m$, then $( \fhi\pi_n^m)\times\psi :A_m\times B_m\to A\times B$ is an epimorphism; otherwise, $ \fhi \times (\psi\rho_m^n):A_n\times B_n\to A\times B$ is.

Let now $E_1\times E_2$, $F_1\times F_2\in \mathcal F $, $n\in \N $, with epimorphisms $ \fhi_1\times \fhi_2:F_1\times F_2\to E_1\times E_2$, $\psi_1\times\psi_2:A_n\times B_n\to E_1\times E_2$.
Let $m,m'\geq n$ with epimorphisms $\theta_1:A_m\to F_1$, $\theta_2:B_{m'}\to F_2$ be such that $ \fhi_1\theta_1=\psi_1\pi_n^m$, $ \fhi_2\theta_2=\psi_2\rho_n^{m'}$.
Suppose for instance that $m\leq m'$.
Then $( \fhi_1\times \fhi_2)((\theta_1\pi_m^{m'})\times\theta_2)=(\psi_1\times\psi_2)(\pi_n^{m'}\times\rho_n^{m'}):A_{m'}\times B_{m'}\to E_1\times E_2$.
\qed

\medskip
So $ \mathbb F= \mathbb F_1\times \mathbb F_2$ is the support of the projective Fra\"iss\'e limit of $ \mathcal F $.
Moreover, denoting $\pi_n^{\infty }: \mathbb F_1\to A_n$, $\rho_n^{\infty }: \mathbb F_2\to B_n$ the projections of the limits onto the members of the fundamental sequences, and given $(a,b),(a',b')\in \mathbb F $,
\[
\begin{split}
& (a,b)R^{ \mathbb F }(a',b')\Leftrightarrow\forall n\in \N \ (\pi_n^{\infty }(a),\rho_n^{\infty }(b))R^{A_n\times B_n}(\pi_n^{\infty }(a'),\rho_n^{\infty }(b'))\Leftrightarrow \\
& \Leftrightarrow\forall n\in \N \ (\pi_n^{\infty }(a)R^{A_n}\pi_n^{\infty }(a')\wedge\rho_n^{\infty }(b)R^{B_n}\rho_n^{\infty }(b'))\Leftrightarrow aR^{ \mathbb F_1}a'\wedge bR^{ \mathbb F_2}b'.
\end{split}
\]
So $ \quot{ \mathbb F}{R^{ \mathbb F }}$ is homeomorphic to $ \quot{ \mathbb F_1}{R^{ \mathbb F_1}}\times \quot{ \mathbb F_2}{R^{ \mathbb F_2}}$.
\end{proof}

\begin{theorem} \label{identification}
Let $X$ be a finitely representable metric space; say this is witnessed by a language $ \mathcal L_R$ and a homeomorphism $\Phi :X\to \quot{ \mathbb F }{R^{ \mathbb F }} $.
Let $G$ be the set of all almost stable points of $ \mathbb F $.
Let $\equiv $ be a closed equivalence relation on $ \mathbb F $ such that $R^{ \mathbb F }\subseteq {\equiv } , {\equiv }\setminus G^2=R^{ \mathbb F }\setminus G^2$.
Let $\cong $ be the equivalence relation defined on $X$ by letting
\[
x\cong y\Leftrightarrow\exists u,v\in \mathbb F\ (u\equiv v\wedge\pi (u)=\Phi (x)\wedge\pi (v)=\Phi (y))
\]
where $\pi : \mathbb F \to \quot{ \mathbb F }{R^{ \mathbb F }} $ is the quotient map.

Then $X'= \quot X{\cong } $ is finitely representable.
\end{theorem}

\begin{proof}
By proposition \ref{machrelassional} and the remark following it, we can assume that $ \mathcal L_R$ is a relational language.
Let $ \mathcal F $ be a projective Fra\"iss\'e family of finite topological $ \mathcal L_R$-structures of which $ \mathbb F $ is a projective Fra\"iss\'e limit.

Notice that since $G$ is $R^{ \mathbb F }$-invariant, it is also $\equiv $-invariant.

\begin{claim} \label{eightone}
Assume that $A\in \mathcal F $ and let $ \fhi ,\psi: \mathbb F \to A$ be $ \mathcal L_R$-epimorphisms.
Then $\fhi \! \times \! \fhi \,(\equiv )=\psi \! \times \! \psi \,(\equiv )$, that is, if $a,b\in A$ then there are $u,v\in \mathbb F $ such that $ \fhi (u)=a$, $ \fhi (v)=b$, $u\equiv v$ if and only if there are $u',v'\in \mathbb F $ such that $\psi (u')=a$, $\psi (v')=b$, $u'\equiv v'$.
\end{claim}

{\it Proof of claim.}
Let $\alpha : \mathbb F \to \mathbb F $ be an isomorphism such that $ \fhi =\psi\alpha $.
Assume that $a,b\in A$, $u,v\in \mathbb F $ are such that $ \fhi (u)=a$, $ \fhi (v)=b$, $u\equiv v$.
Denote $u'=\alpha (u)$, $v'=\alpha (v)$, so that $\psi (u')=a$, $\psi (v')=b$.
By the $\equiv $-invariance of $G$, we have that either $u,v$ are both in $G$ or they are both outside $G$.
If $u,v\notin G$, then $uR^{ \mathbb F }v$, so that $u'R^{ \mathbb F }v'$ and consequently $u'\equiv v'$.
If instead $u,v\in G$, then $u'R^{ \mathbb F}u\equiv vR^{ \mathbb F }v'$ and again $u'\equiv v'$ and we are done.
\qed

\medskip
Set $ \mathcal L'_S= \mathcal L_R\cup\{ S\} $, where $S$ is a new binary relation symbol.
For every $A\in \mathcal F $ let $A'$ be the expansion of $A$ to $ \mathcal L'_S$ defined by letting $S^{A'}= \fhi \! \times \!  \fhi \,(\equiv )$ for any arbitrary $ \mathcal L_R$-epimorphism $ \fhi : \mathbb F \to A$.
Let $ \mathcal F'=\{ A'\}_{A\in \mathcal F }$.

\begin{claim}
Given $A,B\in \mathcal F $, a function $ \fhi :A\to B$ is an $ \mathcal L_R$-epimorphism if and only if it is an $ \mathcal L'_S$-epimorphism from $A'$ to $B'$.
\end{claim}

{\it Proof of claim.}
The backward implication holds as $A',B'$ are expansions of $A,B$, respectively.

For the forward direction, it is enough to show that $ \fhi $ respects $S$.
So let $a,b\in A$ be such that $aS^{A'}b$; pick any $ \mathcal L_R$-epimorphism $\psi : \mathbb F \to A$ and let $u,v\in \mathbb F $ be such that $\psi (u)=a$, $\psi (v)=b$, $u\equiv v$.
So, by claim \ref{eightone}, $u,v$, together with the $ \mathcal L_R$-epimorphism $ \fhi \psi : \mathbb F \to B$, witness that $ \fhi (a)S^{B'} \fhi (b)$.
Conversely, let $a,b\in B$ be such that $aS^{B'}b$ and fix an arbitrary $ \mathcal L_R$-epimorphism $\psi : \mathbb F \to B$; then there are $u,v\in \mathbb F $ such that $a=\psi (u)$, $b=\psi (v)$, $u\equiv v$.
Let $\theta : \mathbb F \to A$ be an $ \mathcal L_R$-epimorphism such that $ \fhi \theta =\psi $; such an epimorphism exists by combining (L1) and (L3).
Then, again by claim \ref{eightone}, $\theta (u)S^{A'}\theta (v)$, $ \fhi \theta (u)=a$, $ \fhi \theta (v)=b$ and we are done.
\qed

\medskip
By the claim, $ \mathcal F'$ is a projective Fra\"iss\'e family and a projective Fra\"iss\'e limit $ \mathbb F'$ of $ \mathcal F'$ is an expansion of $ \mathbb F $ to $ \mathcal L'_S$.
As for the interpretation of $S$ in $ \mathbb F'$, we have the following.

\begin{claim}
$S^{ \mathbb F'}={\equiv} $.
\end{claim}

{\it Proof of claim.}
Let $u,v\in\mathbb F'$ and assume first $uS^{ \mathbb F'}v$.
By the closure of $\equiv $, to show $u\equiv v$ it is enough to prove that for any clopen neighbourhoods $U,V$ of $u,v$, respectively, there are $u'\in U$, $v'\in V$ with $u'\equiv v'$, where we can take $U=V$ if $u=v$, and $U\cap V=\emptyset $ otherwise.
So let $A'\in \mathcal F'$ with an epimorphism $ \fhi : \mathbb F'\to A'$ refining the clopen covering $\{U,V, \mathbb F'\setminus(U\cup V)\} $.
Since $ \fhi (u)S^{A'} \fhi (v)$, there are $u',v'\in \mathbb F'$ with $ \fhi (u')= \fhi (u)$, $ \fhi(v')= \fhi (v)$, $u'\equiv v'$. Since it follows that $u'\in U,v'\in V$, we are done.

Conversely, suppose $u\equiv v$.
Again, fix any clopen neighbourhoods $U,V$ of $u,v$, respectively, such that $U=V$ if $u=v$, and $U,V$ disjoint otherwise.
Pick $A'\in \mathcal F'$ and an epimorphism $ \fhi : \mathbb F' \to A'$ refining the clopen covering $\{ U,V, \mathbb F' \setminus (U\cup V)\} $.
Since $ \fhi (u)S^{A'} \fhi (v)$, there are $u',v'\in \mathbb F'$ (actually $u'\in U,v'\in V$) with $u'S^{ \mathbb F'}v'$, and we are done again.
\qed

\medskip
To finish the proof, notice that $X'$ is homeomorphic to $ \quot{ \mathbb F }{\equiv } $.
\end{proof}

\section{Arcs, hypercubes, graphs} \label{graphs}
We now apply the results of the preceding sections to demonstrate the finite representability of some classes of continua.
We begin by establishing the following.

\begin{theorem} \label{repnarcs}
Arcs are finitely representable.
\end{theorem}

We prove theorem \ref{repnarcs} through a sequence of lemmas.

Let $ \mathcal L_R=\{ R,\leq\} $, where $\leq $ is a binary relation symbol.
Let $ \mathcal X $ be the class of those finite topological $ \mathcal L_R$-structures $A$ such that:
\begin{itemize}
\item $\leq^A$ is a total order;
\item $aR^Ab$ if and only if $a=b$ or $a,b$ are $\leq^A$-consecutive.
\end{itemize}

\begin{lemma}
Class $ \mathcal X $ is a projective Fra\"iss\'e family.
\end{lemma}

\begin{proof}
If $A=\{ 1\}\in \mathcal X $ is defined by letting $R^A= {\leq^A} =\{ (1,1)\} $, then for any $B\in \mathcal X $ the constant map $ \fhi :B\to A$ is an epimorphism.
So it is enough to verify (AP).

Let $A,B,C\in \mathcal X $ with epimorphisms $ \fhi :B\to A$, $\psi :C\to A$.
Let 
\[
a_1\leq^A\ldots\leq^Aa_{ \card (A)}
\]
be an enumeration of $A$.
Let $N_j=\max ( \card ( \fhi^{-1}(\{ a_j\} )), \card (\psi^{-1}(\{ a_j\} )))$, for each $j\in\{ 1,\ldots , \card (A)\} $, and define $D\in \mathcal X $ such that
\[
\card (D)=\sum_{j=1}^{ \card (A)}N_j
\]
and enumerate it as $D=\{ d_{jl}\mid j\in\{ 1,\ldots , \card (A)\} ,l\in\{ 1,\ldots ,N_j\}\}$.
Let $\leq^D$ be the total order on $D$ determined by the lexicographic order on the pairs of indices $(j,l)$.
This determines relation $R^D$ too.

Now define $\theta :D\to B$ by mapping $\{ d_{j1},\ldots ,d_{jN_j}\} $ onto $ \fhi^{-1}(\{ a_j\} )$ in an increasing way, and similarly define $\rho :D\to C$.
So $\theta ,\rho $ are epimorphisms and $ \fhi \theta =\psi\rho $.
\end{proof}

Let $ \mathbb X $ be the projective Fra\"iss\'e limit of $ \mathcal X $.

\begin{lemma} \label{ordinconminemass}
Relation $\leq^{ \mathbb X }$ is a total order on $ \mathbb X $ having a least and a last element.
\end{lemma}

\begin{proof}
By proposition \ref{arflessivsimetric}, parts (1)(3)(4)(5)(6).
\end{proof}

\begin{lemma}
Relation $R^{ \mathbb X }$ is an equivalence relation.
\end{lemma}

\begin{proof}
By proposition \ref{arflessivsimetric}, parts (1)(2), $R^{ \mathbb X }$ is reflexive and symmetric.
To complete the proof, it will be shown that every $x\in \mathbb X $ is $R^{ \mathbb X }$-related to at most one element different from itself.

\sloppypar
So suppose, towards a contradiction, that $x,y_1,y_2$ are distinct elements in $ \mathbb X $ such that $y_1R^{ \mathbb X }xR^{ \mathbb X }y_2$.
Let $U,V_1,V_2$ be disjoint clopen neighbourhoods of $x,y_1,y_2$, respectively.
If $ \fhi : \mathbb X \to A$ is any epimorphism onto an element of $ \mathcal X $ refining $\{ U,V_1,V_2, \mathbb X \setminus (U\cup V_1\cup V_2)\} $, since $ \fhi (x)$, $ \fhi (y_1)$, $ \fhi (y_2)$ are distinct and $ \fhi (y_1)R^A \fhi (x)R^A \fhi (y_2)$, it follows that $ \fhi (y_1), \fhi (x), \fhi (y_2)$ are $\leq^A$-consecutive, with $ \fhi (x)$ being the midpoint.
Say, for instance, $ \fhi (y_1)\leq^A \fhi (x)\leq^A \fhi (y_2)$.
Then let $B=A\cup\{ z\} $, where $z\notin A$, with the symbols of $ \mathcal L_R$ interpreted as follows:
\begin{itemize}
\item $\leq^B$ is obtained from $\leq^A$ by inserting $z$ between $ \fhi (x), \fhi (y_2)$;
\item $R^B$ is the only extension of $R^A$ compatible with the definition of $\leq^B$ that turns $B$ in an element of $ \mathcal X $.
\end{itemize}
Define $\psi :B\to A$ as the identity on the elements of $A$ and by letting $ \psi (z)= \fhi (x)$.
Then there cannot be any epimorphism $\theta : \mathbb X \to B$ such that $ \fhi =\psi\theta $, since $ \theta (x)$ could not be $R^B$-related to both $ \fhi (y_1), \fhi (y_2)$.
\end{proof}

\begin{lemma} \label{antervajbomba}
If $x\in \mathbb X $ then $x$ has a basis of clopen neighbourhoods that are convex sets with respect to $\leq^{ \mathbb X }$.
\end{lemma}

\begin{proof}
Let $U$ be a clopen subset of $ \mathbb X $ containing $x$.
Let $ \fhi : \mathbb X \to A$ be an epimorphism onto some $A\in \mathcal X $ refining the clopen covering $\{ U, \mathbb X \setminus U\} $.
Let $V= \fhi^{-1}(\{ \fhi (x)\} )$, so that $V$ is clopen.
If $y,z\in V$ with $y\leq^{ \mathbb X }z$, then for any $w\in \mathbb X $ with $y\leq^{ \mathbb X }w\leq^{ \mathbb X }z$ one has $ \fhi (w)= \fhi (x)$, whence $w\in V$.
\end{proof}

\begin{lemma} \label{consecutiv}
If $x,y\in \mathbb X $, then $x,y$ are $\leq^{ \mathbb X }$-consecutive if and only if they are distinct and $R^{ \mathbb X }$-related.
\end{lemma}

\begin{proof}
Suppose $x\leq^{ \mathbb X }y$, so that in particular $ \fhi (x)\leq^A \fhi (y)$ for any epimorphism $ \fhi $ from $ \mathbb X $ onto some $A\in \mathcal X $.

Assume first they are consecutive (in particular, $x\neq y$).
First, notice that for any $A\in \mathcal X $ and epimorphism $ \fhi : \mathbb X \to A$ either $ \fhi (x)= \fhi (y)$ or $ \fhi (x), \fhi (y)$ are $\leq^A$-consecutive, since $ \fhi $ is monotone with respect to the orders.
So it follows that $ \fhi (x)R^A \fhi (y)$.
By the arbitrarity of $A$ and $ \fhi $, this implies $xR^{ \mathbb X }y$.

Conversely, assume $x\neq y$, $xR^{ \mathbb X }y$ and suppose there is $z\in \mathbb X $ with $x<^{ \mathbb X }z<^{ \mathbb X }y$.
Let $U,V,W$ be disjoint clopen neighbourhoods of $x,y,z$, respectively.
Let $A\in \mathcal X$ with an epimorphism $ \fhi : \mathbb X \to A$ refining $\{ U,V,W, \mathbb X \setminus (U\cup V\cup W)\} $.
Then $ \fhi (x)<^A \fhi (z)<^A \fhi (y)$, so $ \fhi (x), \fhi (y)$ are not $R^A$-related, a contradiction.
\end{proof}

\begin{lemma} \label{ordincomplet}
A closed total order $\leq $ on a compact metric space $X$ is complete.
\end{lemma}

\begin{proof}
Let $A$ be a bounded non-empty subset of $X$.
Let $A'=\{ x\in X\mid\forall y\in A\ y\leq x\} $, the set of upper bounds of $A$, which is a closed non-empty subset of $X$.
It is then enough to establish the existence of $\min A'$.
Let $\{ x_{\alpha }\}_{\alpha\in\beta }$ be a maximal decreasing sequence in $A'$.
Since every $\leq $-open interval is an open subset of $X$, by separability of $X$ ordinal $\beta $ must be countable.
If $\beta =\gamma +1$ is a successor ordinal, then $x_{\gamma }=\min A'$.
Otherwise, by compactness, $\inf\{ x_{\alpha }\}_{\alpha\in\beta }$ exists and it equals $\min A'$.
\end{proof}

Let $Q= \quot{ \mathbb X }{R^{ \mathbb X }} $ and let $\pi : \mathbb X \to Q$ be the quotient map.
On $Q$ define $[x]\leq' [y]$ if and only if $x\leq^{ \mathbb X }y$.
By lemma \ref{consecutiv} this is well defined.
Moreover, by lemmas \ref{consecutiv}, \ref{ordincomplet} and \ref{ordinconminemass}, this is a dense, complete total order with a first and a last element.

\begin{lemma} \label{ordertopology}
The quotient topology on $Q$ is the order topology induced by $\leq'$.
\end{lemma}

\begin{proof}
We first show that sets of the form $I_{[a]}=\{ [x]\in Q\mid [a]<'[x]\} $, $I^{[b]}=\{ [x]\in Q\mid [x]<'[b]\} $ are open in $Q$.
For the first kind, since $[a]$ contains at most two elements, let $a^*$ be its maximum with respect to $\leq^{ \mathbb X }$.
Then $I_{[a]}$ is the image under $\pi $ of $\{ x\in \mathbb X \mid a^*<^{ \mathbb X }x\} $, which is open (since $\leq^{ \mathbb X }$ is closed and total) and $R^{ \mathbb X }$-invariant.
The same argument works for the second type of intervals.

Conversely, let $U$ be open in $Q$ and fix $[x]\in U$.
By lemma \ref{antervajbomba} for each point in $[x]$ there is a $\leq^{ \mathbb X }$-convex, clopen subset of $ \mathbb X $ containing that point and included in $\pi^{-1}(U)$.
Since $[x]$ is either a singleton or a doubleton consisting of two $\leq^{ \mathbb X }$-consecutive points, the union of these clopen sets, call it $I$, is $\leq^{ \mathbb X }$-convex.
It is then enough to show that, if $\min Q\neq [x]$, then $I$ contains some element that strictly precedes all elements of $[x]$, and similarly that if $\max Q\neq [x]$ then $I$ contains some element strictly bigger than the elements of $[x]$.
So suppose $\min Q\neq [x]$.
If, towards a contradiction, $[x]$ contained the least element of $I$, let $J$ be the set of all strict predecessors of $\min I$.
Since $I$ is clopen and $\leq^{ \mathbb X }$ is closed, $J$ is a clopen, non-empty, bounded subset of $ \mathbb X $.
By lemma \ref{ordincomplet}, $J$ has a maximum $z$.
So $z$ is an immediate predecessor of $\min I$, but $z$ and $\min I$ are not $R^{ \mathbb X }$-related, since $\min I\in [x]\subseteq I$.
This contradicts lemma \ref{consecutiv}.
\end{proof}

\begin{lemma}
$\leq'$ has order type $1+\lambda +1$, where $\lambda $ is the order type of the real line.
\end{lemma}

\begin{proof}
We already noted that $\leq'$ is bounded and complete.
We remark that it is also a separable order: indeed, it is a dense order, so every open interval is non-empty and, by lemma \ref{ordertopology}, open in the Polish space $Q$, thus every interval contains a point of a fixed countable dense subset of $Q$.
Now apply \cite[theorem 2.30]{Rosens1982}.
\end{proof}

Since the topology of $Q$ is induced by an order of type $1+\lambda +1$, it follows that $Q$ is an arc, concluding the proof of theorem \ref{repnarcs}.

An immediate consequence is now the following.
Recall that a hypercube is a space homeomorphic to $[0,1]^n$, for some $n$.

\begin{corollary}
Every hypercube is finitely representable.
\end{corollary}

\begin{proof}
By theorems \ref{repnarcs} and \ref{products}.
\end{proof}

For the next consequence recall that, in continuum theory, a graph is defined as a finite union of arcs any two of them meeting at most in one or both of their endpoints (see for example \cite{Nadler1992}).

\begin{corollary}
Every graph is finitely representable.
\end{corollary}

\begin{proof}
Notice that in the proof of theorem \ref{repnarcs} each endpoint of arc $Q$ is the image under the quotient map of an almost stable point, since the extrema of a total order --- in this case $ \leq^{ \mathbb X }$ --- are preserved under isomorphism.
So we can use theorem \ref{closurefinitesums} to obtain a disjoint union of arcs; the remark following that theorem allows us to apply theorem \ref{identification} to glue endpoints and thus obtain any possible graph.
\end{proof}

\section{Questions} \label{qs}

In the previous sections we exhibited some simple classes of finitely representable spaces, enlarging the examples given in \cite{Camerl2010}.
This suggests the following general question.

\begin{question}
What spaces are finitely representable?
\end{question}

In our examples, due to the application of the constructions of section \ref{sum-prod}, the languages and the structures associated to the spaces were in some sense always related to the obvious structural characteristics of the spaces, starting from an order representing the arc.
The following rather vague question comes to mind.

\begin{question}
Given a finitely representable space, what are the minimal, or most natural, language and structures representing it?
Can some specific features of the space be derived directly from the language?
What are the obstructions that forbid a space to be represented with a given language?
\end{question}

\end{document}